\documentclass[11pt]{amsart}

\pdfoutput=1

\usepackage[text={420pt,660pt},centering]{geometry}

\usepackage{comment}
\usepackage{resmes}
\usepackage{yfonts}
\usepackage[foot]{amsaddr} 

\usepackage{cancel}
\usepackage{esint,amssymb} 
\usepackage{graphicx}

\usepackage{MnSymbol}
\usepackage{mathtools} 

\usepackage[colorlinks=true, pdfstartview=FitV, linkcolor=blue, citecolor=blue, urlcolor=blue,pagebackref=false]{hyperref}
\usepackage{orcidlink}
\usepackage{microtype}
\usepackage{bm}
\usepackage{mathrsfs}
\usepackage{xcolor}
\usepackage{enumitem}

\definecolor{darkgreen}{rgb}{0,0.5,0}
\definecolor{darkblue}{rgb}{0,0,0.7}

\newtheorem{proposition}{Proposition}
\newtheorem{theorem}[proposition]{Theorem}
\newtheorem{lemma}[proposition]{Lemma}

\theoremstyle{remark}
\newtheorem{remark}[proposition]{Remark}

\theoremstyle{definition}

\numberwithin{equation}{section}
\numberwithin{proposition}{section}
\numberwithin{figure}{section}
\numberwithin{table}{section}

\newcommand{\Z}{\mathbb{Z}}
\newcommand{\N}{\mathbb{N}}

\newcommand{\R}{\mathbb{R}}

\newcommand{\E}{\mathbb{E}}

\newcommand{\eps}{\varepsilon}

\renewcommand{\leq}{\leqslant}
\renewcommand{\geq}{\geqslant}

\renewcommand{\subset}{\subseteq}
\renewcommand{\bar}{\overline}
\renewcommand{\tilde}{\widetilde}

\newcommand{\Ll}{\left}
\newcommand{\Rr}{\right}
\renewcommand{\d}{\mathrm{d}}

\newcommand{\mcl}{\mathcal}

\newcommand{\mfk}{\mathfrak}

\newcommand{\la}{\left\langle}
\newcommand{\ra}{\right\rangle}

\renewcommand{\H}{\mathsf{H}}

\renewcommand{\S}{S}

\newcommand{\D}{{D}}

\newcommand{\cL}{{\mathcal{L}}}

\newcommand{\mmse}{\mathrm{mmse}}
\newcommand{\MMSE}{\mathrm{MMSE}}
\newcommand{\sfp}{{\mathsf{p}}}
\newcommand{\bx}{\mathbf{x}}
\newcommand{\bn}{{\mathbf{n}}}
\newcommand{\bd}{{\mathbf{d}}}
\newcommand{\cD}{\mathcal{D}}
\newcommand{\overlap}{Q}

\begin{document}

\author{Hong-Bin Chen\,\orcidlink{0000-0001-6412-0800}}
\address[Hong-Bin Chen]{Institut des Hautes Études Scientifiques, Bures-sur-Yvette, France}
\email{\href{mailto:hongbin.chen@nyu.edu}{hongbin.chen@nyu.edu}}

\author{Victor Issa\,\orcidlink{0009-0009-1304-046X}}
\address[Victor Issa]{Department of Mathematics, ENS de Lyon, Lyon, France}
\email{\href{mailto:victor.issa@ens-lyon.fr}{victor.issa@ens-lyon.fr}}

\subjclass[2010]{}
\date{}

\title[Differentiability and overlap concentration]{Differentiability and overlap concentration in optimal Bayesian inference}

\begin{abstract}
In this short note, we consider models of optimal Bayesian inference of finite-rank tensor products. We add to the model a linear channel parametrized by $h$. We show that at every interior differentiable point $h$ of the free energy (associated with the model), the overlap concentrates at the gradient of the free energy and the minimum mean-square error converges to a related limit. In other words, the model is replica-symmetric at every differentiable point. At any signal-to-noise ratio, such points $h$ form a full-measure set (hence $h=0$ belongs to the closure of these points). For a sufficiently low signal-to-noise ratio, we show that every interior point is a differentiable point.

    \bigskip


    \medskip

\end{abstract}

\maketitle

\tableofcontents

%
%
%
%
%
%

\section{Introduction}

Recently, there has been tremendous progress in understanding the information-theoretical aspect of statistical inference models. These works often utilize the toolbox from the mean-field spin glass theory. In particular, the Bayesian inference of finite-rank tensor products in the optimal case (where the posterior is known) can be seen as a simple variant of a mean-field spin glass model. One important feature of the optimal model is that the system is always in the replica symmetry regime. This means that the spin-glass order parameter, the overlap, always concentrates (under a small perturbation).

In this note, we clarify the connection between the differentiability of the limit free energy associated with the inference model and the concentration of the overlap as well as the convergence of minimal mean-square errors. The theme is closely related to the so-called generic model in the spin glass (see~\cite[Section~15.4]{Tbook2} and~\cite[Section~3.7]{pan}). 
In the context of optimal Bayesian inference, we show that at any differentiable point (with respect to the parameter $h$ for an additional linear channel (or external field)) of the limit free energy, the overlap concentrates. 
In other words, the model is
replica-symmetric at such points.
This is different from the result by Barbier in~\cite{barbier2021overlap} that shows the concentration of the overlap under an average of small perturbation (which corresponds to the same phenomenon in general non-generic spin glass models).

The parameter $h$ also corresponds to the spatial parameter in the Hamilton--Jacobi equation representation of the limit free energy. 
Since the limit free energy is Lipschitz, Lebesgue-almost-every point is differentiable.
Moreover, under our assumption, the initial condition (namely, the free energy associated with the linear channel) is smooth. Hence, we can solve the equation using characteristic lines for a short time (small signal-to-noise ratio) and the solution is smooth. Therefore, in this regime, every interior point $h$ is a differentiable point.

This short-time regularity result is less obvious in our case because the equation is not posed on the entire Euclidean space but on a closed convex cone. Hence, it is important that the nonlinearity ``points in the right direction'' so that the characteristics do not leave the domain.

\subsection{Setting and the main result}

Throughout, we write $\R_+=[0,\infty)$ and $\R_{++}=(0,\infty)$.
For any matrices $a,b$ or vectors of the same dimension, we write $a\cdot b = \sum_{ij}a_{ij}b_{ij}$ as the entry-wise inner product and write $|a|=\sqrt{a\cdot a}$. Throughout, we fix $\D\in\N$ and let $\S^\D$ be the linear space of $\D\times\D$ real symmetric matrices. We view $\S^\D$ as a subspace of $\R^{\D\times\D}$ and endow $\S^\D$ with the aforementioned entry-wise inner product. Let $\S^\D_+$ (resp.\ $\S^\D_{++}$) be the subset of $\S^\D$ consisting of positive semi-definite (resp.\ definite) matrices. Notice that $\S^\D_+$ (resp.\ $\S^\D_{++}$) is a closed (resp.\ open) convex cone in $\S^\D$.
We can identify $\S^\D$ isometrically with a Euclidean space and put the Lebesgue measure on it. Throughout, a full-measure subset of $\S^\D_+$ means that its complement in $\S^\D_+$ has Lebesgue measure zero.

We consider the statistical inference problem of general tensor products as in~\cite[Section~1.1]{chen2022statistical}.
Fix $\D$ (which is $K$ in~\cite{chen2022statistical}) to be the rank of the signal. For each $N\in\N$, we denote the $\R^{N\times \D}$-valued signal by $X$. We assume that the distribution of $X$ is known and we denote it by $P_N$.
Fix any $\sfp \in\N$ and we view the tensor product $X^{\otimes p}$ as an $N^\sfp\times \D^\sfp$-matrix in terms of the Kronecker product.
Fix any deterministic $A \in \R^{\D^\sfp\times L}$ for some $L\in \N$. We view $A$ as the matrix describing the interaction of entries in $X^{\otimes\sfp}$. 
Let $t\geq 0$ and we interpret $2t$ as the signal-to-noise ratio.
The $\R^{N^\sfp\times L}$-valued noisy observation $Y$ is given by 
\begin{align*} 
    Y = \sqrt{\frac{2t}{N^{\sfp-1}}} X^{\otimes \sfp}A+W
\end{align*}
where $W$ is an $N^{\sfp}\times L$ matrix with i.i.d.\ standard Gaussian entries.
In addition to $Y$, we also consider an independent $\R^{N\times \D}$-valued linear channel
\begin{align}\label{e.barY=Xsqrt(2h)+Z}
    \bar Y = X\sqrt{2h}+Z
\end{align}
for an $N\times \D$ matrix $Z$ consisting of i.i.d.\ standard Gaussian entries. From the perspective of statistical mechanics, $\bar Y$ gives rise to an external field in the system.

For $N\in\N$, $t\in\R_+$, and $h\in \S^\D_+$,
we consider the random Hamiltonian
\begin{align}\label{e.H_N(t,h,x)}
\begin{split}
    H_N(t,h,\bx) = &\sqrt{\frac{2t}{N^{\sfp-1}}}\Ll(\bx^{\otimes p}A\Rr)\cdot Y - \frac{t}{N^{\sfp-1}}\Ll|\bx^{\otimes \sfp}A\Rr|^2
    \\
    &+\sqrt{2h}\cdot \Ll(\bx^\intercal \bar Y\Rr)- h\cdot \Ll(\bx^\intercal \bx\Rr).
\end{split}
\end{align}
Here, $\sqrt{h}$ is the matrix square root of $h$, which is well-defined for $h\in\S^\D_+$.
The randomness in $H_N(t,h,\bx)$ comes from $X$, $W$, and $Z$.
We define the associated Gibbs measure by
\begin{align}\label{e.<>=}
    \la\cdot\ra_{N,t,h} \propto \exp\Ll(H_N(t,h,\bx)\Rr) \d P_N(\bx).
\end{align}
By the Bayes rule, the Gibbs measure satisfies
\begin{align}\label{e.<g(x)>=}
    \la g \Ll(\bx\Rr) \ra_{N,t,h} = \E \Ll[g\Ll(X\Rr)\,\big|\,Y,\bar Y\Rr]
\end{align}
for any bounded measurable function $g$.
Slightly abusing the notation, we denote still by $\la\cdot\ra_{N,t,h}$ its tensorized version, which allows us, for instance, to consider independent samples $\bx$ and $\bx'$ from $\la\cdot\ra_{N,t,h}$.

let $F_N(t,h)$ be the enriched free energy (see~\cite[(1.3)]{chen2022statistical}) defined by
\begin{align}\label{e.F_N=}
    F_N(t,h) = \frac{1}{N}\log\int_{\R^{N\times D}}\exp\Ll(H_N(t,h,\bx)\Rr)\d P_N(\bx).
\end{align}
We write
\begin{align*}
    \bar F_N(t,h)=\E F_N(t,h)
\end{align*}
where $\E$ averages over the randomness of $X$, $W$, and $Z$. As a consequence, $\bar  F_N(t,h)$ is nonrandom.
We view $\bar F_N$ as a real-valued function on $\R_+\times \S^\D_+$.

We often impose some of the following assumptions:
\begin{enumerate}[start=1,label={\rm (H\arabic*)}]
    \item \label{i.assume_1_support_X}
    For each $N\in\N$, every entry in $X$ is in $[-1,+1]$ a.s.\ under $P_N$.
    \item \label{i.assume_2_F_N(0,0)_cvg}
    As $N\to\infty$, $\Ll(\bar F_N(0,\cdot)\Rr)_{N\in\N}$ converges pointwise everywhere to some continuously differentiable function $\psi:\S^\D_+\to\R$.
    \item \label{i.assume_3_concent}
    For every compact subset $K\subset \R_+\times \S^\D_+$, we have $\lim_{N\to\infty} \E \Ll|F_N-\bar F_N\Rr|^2_{L^\infty(K)}=0$.
\end{enumerate}
A stronger assumption is the following:
\begin{enumerate}[label={\rm (HS)}]
    \item \label{i.assume_S}
    For each $N\in\N$, row vectors in $X$ are i.i.d.\ with a fixed distribution $P_1$. Let $X_1$ be the first row-vector of $X$. Under $P_1$, entries of $X_1$ are in $[-1,+1]$.
\end{enumerate}
Under~\ref{i.assume_S}, \ref{i.assume_1_support_X} clearly holds; \ref{i.assume_2_F_N(0,0)_cvg} holds with $\psi = \bar F_1(0,\cdot)$ because in this case $\bar F_N(0,\cdot)= \bar F_1(0,\cdot)$ for every $N\in\N$; and we can deduce~\ref{i.assume_3_concent} using standard techniques \cite[Lemma~C.1]{chen2022hamilton} (this lemma assumes that $X$ has i.i.d.\ entries but the same argument holds for~\ref{i.assume_S}). 

Define $\H:\S^\D\to\R$ (as in~\cite[(1.4)]{chen2022statistical}) by
\begin{align}\label{e.H(q)=(AA)...}
    \H(q) = \Ll(AA^\intercal\Rr)\cdot q ^{\otimes \sfp},\quad\forall q\in\S^\D.
\end{align}

\begin{theorem}[\cite{chen2022statistical} limit of free energy]\label{t.cvgF_N_to_f}
Assume~\ref{i.assume_1_support_X}, \ref{i.assume_2_F_N(0,0)_cvg}, and~\ref{i.assume_3_concent}.
The function $\bar F_N$ converges pointwise everywhere on $\R_+\times\S^\D_+$ to the unique Lipschitz viscosity solution of
\begin{align}\label{e.hj_stats}
    \partial_t f- \mathsf{H}(\nabla_h f)=0,\quad\text{on $\R_+\times \S^\D_+$}
\end{align}
with initial condition $f(0,\cdot) = \psi$. 

Moreover, $f$ always admits the representation by the Hopf formula:
\begin{align}\label{e.hopf_inference}
    f(t,h)= \sup_{h'\in\S^\D_+}\Ll\{h'\cdot h -\psi^*(h')+t\H\Ll(h'\Rr)\Rr\},\quad\forall (t,h)\in\R_+\times\S^\D_+;
\end{align}
if in addition $\H$ is convex on $\S^\D_+$, then $f$ admits the representation by the Hopf--Lax formula:
\begin{align}\label{e.hopf-lax_inference}
    f(t,h)= \sup_{h'\in\S^\D_+}\Ll\{\psi(h+ h')-t \H^*(h'/t)\Rr\},\quad\forall (t,h)\in\R_+\times\S^\D_+.
\end{align}
\end{theorem}

This is the main result from~\cite{chen2022statistical} combined with other results. We give a detailed explanation in Section~\ref{s.maximizers}. In~\eqref{e.hopf_inference} and~\eqref{e.hopf-lax_inference}, $\psi^*$ and $\H^*$ are monotone convex conjugate defined in~\eqref{e.convex_conjugate}.

\begin{remark}[Almost everywhere differentiability]
\label{r.ae_diff}
Since $\bar F_N$ is Lipschitz with coefficient uniform in $N$ (evident from~\eqref{e.derF_N=} and~\ref{i.assume_1_support_X}), we have that $f$ is Lipschitz and thus $f$ is differentiable almost everywhere on $\R_+\times \S^\D_+$ by Rademacher's theorem. \qed
\end{remark}

Next, we introduce two versions of minimal mean-square errors to be considered here:
\begin{gather}
    \MMSE_N(t,h) =\frac{1}{N} \E \Ll[\Ll(X- \E \Ll[ X\,\big|\,Y,\bar Y\Rr]\Rr)^\intercal\Ll(X- \E \Ll[ X\,\big|\,Y,\bar Y\Rr]\Rr)\Rr],\label{e.MMSE=}
    \\
    \mmse_N(t,h)= \frac{1}{N^\sfp}\E \Ll|X^{\otimes\sfp} A - \E \Ll[X^{\otimes\sfp} A\,\big|\, Y,\bar Y\Rr]\Rr|^2.\label{e.mmse=}
\end{gather}
These are natural measures of performance. 
Here, $\MMSE_N(t,h)$ is the $\S^\D_+$-valued MMSE matrix introduced in~\cite{reeves2019geometry,reeves2019geometryIEEE,reeves2018mutual};
$\mmse_N(t,h)$ is natural in the tensor inference setting, which is similar to the one considered in~\cite[(18)]{reeves2020information}.
In the estimation of low rank symmetric matrices \cite{lelarge2019fundamental} (Theorem~1.1) and low-rank asymmetric matrices \cite{miolane2017fundamental} (Proposition~2), the limit of $\mmse_N(t,0)$ was identified. The convergence of matrix MMSE in the inference of second-order matrix tensor products was studied in~\cite[Section~II.B]{reeves2020information}.

Lastly, for each $N\in\N$, we consider the $\D\times \D$ overlap matrix defined by
\begin{align}\label{e.overlap=}
    \overlap = \frac{1}{N} X^\intercal \bx.
\end{align}
Our main results are summarized as follows. 

\begin{theorem}\label{t}
Assume~\ref{i.assume_1_support_X}, \ref{i.assume_2_F_N(0,0)_cvg}, and~\ref{i.assume_3_concent}. Let $t\in \R_{++}$. Then, at every differentiable point $h \in  \S^\D_{++}$ of $f(t,\cdot)$ (which forms a full-measure subset of $\S^\D_+$), we have that $f(\cdot,h)$ is also differentiable at $t$ and the following holds:
\begin{enumerate}
    \item\label{i.t_1} The gradient $\nabla_hf(t,h)$ is the unique maximizer of the Hopf formula~\eqref{e.hopf_inference} at $(t,h)$.
    \item\label{i.t_2} The overlap concentrates at the value $\nabla_hf(t,h)$, namely,
    \begin{align*}
        \lim_{N\to\infty}\E \la \Ll|\overlap - \nabla_h f(t,h)\Rr|\ra_{N,t,h} = 0.
    \end{align*}
    \item\label{i.t_3} Under the stronger assumption~\ref{i.assume_S}, we have
    \begin{align*}
        \lim_{N\to\infty} \MMSE_N(t,h)  &= \E \Ll[X_1^\intercal X_1\Rr] - \nabla_h f(t,h),
        \\
        \lim_{N\to\infty} \mmse_N(t,h)  & = \Ll(AA^\intercal\Rr)\cdot \Ll(\E \Ll[X_1^\intercal X_1\Rr]\Rr)^{\otimes \sfp} - \partial_t f(t,h).
    \end{align*}
\end{enumerate}
Moreover, if $\psi$ is twice differentiable with bounded derivatives (which is satisfied under~\ref{i.assume_S}), setting
\begin{align}\label{e.L=}
    L = \Ll\|\nabla\H(\nabla\psi)\Rr\|_\mathrm{Lip},
\end{align}
we have that $f$ is twice differentiable everywhere on $[0, L^{-1})\times \S^\D_+$.
\end{theorem}

\begin{proof}
Parts~\eqref{i.t_1}, \eqref{i.t_2}, and~\eqref{i.t_3} follow from Propositions~\ref{p.maximizers}, \ref{p.overlap}, and~\ref{p.mmse}, respectively. The short-time differentiability result is from Lemma~\ref{l.HS_condition} and Proposition~\ref{p.C^2}.
\end{proof}

Results in Propositions~\ref{p.maximizers}, \ref{p.overlap}, and~\ref{p.mmse} are slightly more general than the above.
We clarify that even though the short-time differentiability holds at $(t,h)$ with $h \in \S^\D_+\setminus\S^\D_{++}$, the results in Parts~\eqref{i.t_2} and~\eqref{i.t_3} do not seem trivially extendable to such points, including the most interesting case when $h=0$.

The main result~\cite{barbier2021overlap} gives the concentration of $\overlap$ under the measure $\E_h \E \la \cdot\ra_{N,t,h}$ where $\E_h$ is a local average of $h$ over a shrinking set as $N\to\infty$. Moreover, points in the shrinking set converge to $0$. Our Part~\ref{i.t_2} can be interpreted as a pointwise concentration result without average and we are able to recover the value of the limit overlap, but the differentiable point $h$ has to be away from $0$. 

\subsection{Other related works}
Our representation of the limit free energy is based on the Hamilton--Jacobi equation approach. This perspective was first used by Guerra in~\cite{guerra2001sum} and works along the same line include~\cite{barra2,abarra,barramulti,barra2014quantum,barra2008mean,genovese2009mechanical}, which considered equations in finite-dimensions corresponding to the regime of replica symmetry or finite-step replica symmetry breaking. Mourrat started a more mathematical treatment of the approach in~\cite{mourrat2022parisi,mourrat2020hamilton}. In particular, the regime of full replica symmetry breaking is associated with an infinite-dimensional Hamilton--Jacobi equation~\cite{mourrat2020extending,mourrat2021nonconvex,mourrat2023free,HJ_critical_pts,chen2024free,issa2024existence,issa2024hopf}. Through this approach, optimal Bayesian inference models are studied in~\cite{mourrat2021hamilton,HB1,HBJ,chen2022statistical,chen2023free}. The well-posedness of the Hamilton--Jacobi equations are considered in~\cite{chen2022hamilton,chen2022hamiltonCones,issa2024weaksol,issa2024weak}.

Here, we work in the optimal setting, where the system is replica-symmetric so that the relevant equation and the variational formula for the limit free energy are both finite-dimensional.
Works investigating the limit of free energy or equivalently mutual information include~\cite{barbier2016, lelarge2019fundamental, barbier2019adaptive,miolane2017fundamental, barbier2017layered, kadmon2018statistical, luneau2020high,lelarge2019fundamental, mayya2019mutualIEEE, reeves2019geometryIEEE,reeves2020information,lesieur2017statistical, barbier2019adaptive,luneau2019mutual}.
In the non-optimal setting where the prior and the noise are mismatched, the system is in general no longer replica-symmetric. Works in this harder scenario include~\cite{camilli2022inference,barbier2021performance,barbier2022price,pourkamali2022mismatched,guionnet2023estimating}.

As aforementioned, the interplay between the differentiability of the limit free energy and the convergence of overlap is a key feature in the generic spin glass models. This idea was also used in~\cite{HJ_critical_pts,chen2024free,chen2024parisi,issa2024existence} to study the limit of the overlap. A similar idea was employed in~\cite{chen2023on,chen2024conventional} to understand simpler order parameters, such as the self-overlap and the mean magnetization.

\section{Maximizers of variational representations}\label{s.maximizers}

We start by describing how to combine results from different works to get Theorem~\ref{t.cvgF_N_to_f}.
We start with clarifying the meaning of $\psi^*$ and $\H^*$ in~\eqref{e.hopf_inference} and~\eqref{e.hopf-lax_inference}. For any $(-\infty,\infty]$-valued function $g$ defined on a subset of $\S^\D$ containing $\S^\D_+$, its monotone convex conjugate $g^*:\S^\D\to (-\infty,\infty]$ is defined as
\begin{align}\label{e.convex_conjugate}
    g^*(h)= \sup_{h'\in\S^\D_+}\Ll\{h\cdot h'- g(h')\Rr\},\quad\forall h \in \S^\D.
\end{align}

\begin{proof}[Proof of Theorem~\ref{t.cvgF_N_to_f}]
The main part of the theorem and the Hopf formula are direct consequences of~\cite[Theorem~5.1 and Theorem~1.1]{chen2022statistical}. The Hopf--Lax formula is valid due to~\cite[Proposition~6.2]{chen2022hamiltonCones}. To apply this proposition, we need to verify that $\S^\D_+$ satisfies the so-called Fenchel--Moreau property (see~\cite[Definition~6.1]{chen2022hamiltonCones}) which is verified in~\cite[Proposition~B.1]{HBJ}. Also, we need $\psi$ to be increasing, which follows from the monotonicity of $\bar F_N$ implied by the positivity of $\nabla \bar F_N$ evident from~\eqref{e.derF_N=}. Lastly, we need $\H$ to be increasing on $\S^\D_+$, which is proved in~\cite[Lemma~4.2]{chen2022statistical}.
\end{proof}

As a consequence of the envelope theorem (c.f.\ \cite[Theorem~2.21]{HJbook}), we can deduce the following result on the maximizers of variational formulas. 

\begin{proposition}[Properties of maximizers]\label{p.maximizers}
Assume~\ref{i.assume_1_support_X}, \ref{i.assume_2_F_N(0,0)_cvg}, and~\ref{i.assume_3_concent}. Let $f$ be given as in Theorem~\ref{t.cvgF_N_to_f}. Then, for every $(t,h) \in \R_+\times \S^\D_+$, there is a maximizer $h_\star$ of the Hopf formula (the right-hand side in~\eqref{e.hopf_inference} at $(t,h)$) and the following holds:
\begin{enumerate}[start=1,label={\rm (a\arabic*)}]
    \item 
    if 
    $f(t,\cdot)$ is differentiable at $h$, then $h_\star$ is uniquely given by $h_\star = \nabla_h f(t,h)$ and $f(\cdot,h)$ is also differentiable at $t$ with $\partial_t f(t,h) = \H(\nabla_h f(t,h))$;
    \item 
    if 
    $f(\cdot,h)$ is differentiable at $t$, then $h_\star$ satisfies $\H(h_\star) = \partial_t f(t,h)$.
\end{enumerate}

Under an additional assumption that $\H$ is convex on $\S^\D_+$, for every $(t,h)\in \R_+\times \S^\D_+$, there is a maximizer $h_\lozenge$ of the Hopf--Lax formula (the right-hand side in~\eqref{e.hopf-lax_inference} at $(t,h)$) and the following holds:
\begin{enumerate}[start=1,label={\rm (b\arabic*)}]
    \item \label{i.hopf-lax_case_1} 
    if 
    $f(t,\cdot)$ is differentiable at $h$, then $h_\lozenge$ satisfies $\nabla_h\psi(h+h_\lozenge)  = \nabla_h f(t,h)$;
    \item \label{i.hopf-lax_case_2} 
    if 
    $f(\cdot,h)$ is differentiable at $t$ and $\H^*$ is differentiable, then $h_\lozenge$ satisfies\\ $\frac{\d}{\d t}\Ll(-t\H^*(h_\lozenge/t)\Rr)=\partial_t f(t,h)$.
\end{enumerate}
\end{proposition}

\begin{remark}[Alternative Hopf--Lax formula and results]
When $\H$ is convex on $\S^\D_+$, since $\S^\D_+$ is a cone, one can also rewrite the Hopf--Lax formula~\eqref{e.hopf-lax_inference} as
\begin{align}\label{e.f=Hopf-Lax(...th'...)}
    f(t,h)= \sup_{h'\in\S^\D_+}\Ll\{\psi(h+ th')-t \H^*(h')\Rr\},\quad\forall (t,h)\in\R_+\times\S^\D_+
\end{align}
Using this formula and the same argument, we have that there is a maximizer $h_\blacklozenge$ of~\eqref{e.f=Hopf-Lax(...th'...)} satisfying $\nabla_h\psi(h+th_\blacklozenge)=\nabla_h f(t,h)$ in case~\ref{i.hopf-lax_case_1} and $h_\blacklozenge\cdot \nabla_h\psi(h+th_\blacklozenge) -\H^*(h_\blacklozenge)=\partial_t f(t,h)$ in case~\ref{i.hopf-lax_case_2}.\qed
\end{remark}

\begin{remark}
    When $\H$ is strictly convex on $S^D_+$, $\H^*$ is differentiable on $S^D$.
\end{remark}

\begin{proof}[Proof of Proposition~\ref{p.maximizers}]
We only need to verify the existence of maximizers and that they are contained in a fixed compact set when $(t,h)$ varies in a bounded set. The rest follows from the envelope theorem (we use the version in~\cite[Theorem~2.21]{HJbook} which is stated on the entire Euclidean space but can be straightforwardly adapted to $\S^\D_+$ here). 

For the Hopf formula in~\eqref{e.hopf_inference}, since $\psi$ is Lipschitz as argued in Remark~\ref{r.ae_diff}, we can see that $\psi^*$ is equal to $+\infty$ outside a bounded set $K$ independent of $(t,h)$. Since the function $h'\mapsto h'\cdot h-\psi^*(h')+t\H(h')$ is upper semi-continuous, we can deduce the existence of a maximizer, which lies in $K$ independent of $(t,h)$.

For the Hopf--Lax formula in~\eqref{e.hopf-lax_inference}, since $\H$ is bounded on each centered ball, we can see that $\H^*$ grows super linearly in the sense that $\H^*(h')/|h'|$ diverges to $+\infty$ as $|h'|$ grows. This along with the Lipschitzness of $\psi$ implies that maximizers exist and they are contained in a compact set $K_{t,h}$. Also, it is easy to see that if $(t,h)$ varies in a small neighborhood, then $K_{t,h}$ is contained in some fixed compact set.
\end{proof}

\section{Limits of minimal mean square errors}

Recall $\MMSE_N(t,h)$ and $\mmse_N(t,h)$ defined in~\eqref{e.MMSE=} and~\eqref{e.mmse=}.

\begin{proposition}[Limit of MMSE]\label{p.mmse}
Assume~\ref{i.assume_S}.
Let $(t,h)\in\R_+\times \S^\D_+$.
If $h\in \S^\D_{++}$ and $f(t,\cdot)$ is differentiable at $h$, then
\begin{align}\label{e.lim_MMSE}
    \lim_{N\to\infty} \MMSE_N(t,h)  = \E \Ll[X_1^\intercal X_1\Rr] - \nabla_h f(t,h). 
\end{align}
If $t\in \R_{++}$ and $f(\cdot,h)$ is differentiable at $t$, then
\begin{align}\label{e.lim_mmse}
    \lim_{N\to\infty} \mmse_N(t,h)  = \Ll(AA^\intercal\Rr)\cdot \Ll(\E \Ll[X_1^\intercal X_1\Rr]\Rr)^{\otimes \sfp} - \partial_t f(t,h). 
\end{align}
Moreover, the limits in~\eqref{e.lim_MMSE} and~\eqref{e.lim_mmse} are related to the maximizers of variational formulas~\eqref{e.hopf_inference} and~\eqref{e.hopf-lax_inference} via Proposition~\ref{p.maximizers}.
\end{proposition}

\begin{proof}
We need the convexity of $\bar F_N$ proven in~\cite[Lemma~2.3]{chen2022statistical}, which also implies the convexity of $f$.
Fix $(t,h)\in \R_{+}\times \S^\D_{+}$, since $\bar F_N$ is differentiable and $\bar F_N$ converges to $f$ pointwise, it is a classical result \cite[Theorem~25.7]{rockafellar1970convex} in convex analysis that
\begin{gather}
    \text{$h\in\S^\D_{++}$ and $f(t,\cdot)$ is differentiable at $h$}\quad \Longrightarrow\quad \lim_{N\to\infty}\nabla_h \bar F_N(t,h)=\nabla_h f(t,h);\label{e.limdhF_N}
    \\
    \text{$t\in\R_{++}$ and $f(\cdot,h)$ is differentiable at $t$}\quad \Longrightarrow\quad \lim_{N\to\infty}\partial_t \bar F_N(t,h)=\partial_t f(t,h). \label{e.limdtF_N}
\end{gather}
Since we have fixed $(t,h)$ and the value of $N$ is clear from the context, we write $\la\cdot\ra=\la\cdot\ra_{N,t,h}$ to simplify the notation in~\eqref{e.<g(x)>=}.
We also recall the following computation of derivatives from \cite[(2.1) and (2.2)]{chen2022statistical}:
\begin{align}\label{e.derF_N=}
    \partial_t \bar F_N(t,h)= \frac{1}{N^\sfp}\E \Ll|\la\bx^{\otimes \sfp}A\ra\Rr|^2,\qquad \nabla_h\bar F_N(t,h) = \frac{1}{N}\E \Ll[\la \bx\ra^\intercal\la \bx\ra\Rr].
\end{align}

Let us first prove~\eqref{e.lim_mmse}. For simplicity we write $\tilde X = X^{\otimes \sfp}A$ and $\tilde \bx = \bx^{\otimes \sfp}A$. We start with
\begin{align*}
    N^\sfp\, \mmse_N(t,h) \stackrel{\eqref{e.mmse=}}{=}\E \Ll|\tilde X\Rr|^2-\E \Ll|\E\Ll[\tilde X\ \big|\ Y,\bar Y\Rr]\Rr|^2\stackrel{\eqref{e.<g(x)>=}}{=} \E \Ll|\tilde X\Rr|^2 - \E \Ll|\la \tilde\bx\ra\Rr|^2 
    \\
    \stackrel{\eqref{e.derF_N=}}{=}\E \Ll|\tilde X\Rr|^2 - N^\sfp\partial_t \bar F_N(t,h).
\end{align*}
Hence,~\eqref{e.lim_mmse} follows from~\eqref{e.limdtF_N} and
\begin{align}\label{e.N^-pE|X|^2}
    \lim_{N\to\infty}N^{-\sfp}\E \Ll|\tilde X\Rr|^2  = \Ll(AA^\intercal\Rr)\cdot \Ll(\E \Ll[X_1^\intercal X_1\Rr]\Rr)^{\otimes \sfp}
\end{align}
which we explain in the following. We write $\bn =(n_1,\dots ,n_\sfp)\in \{1,\dots,N\}^\sfp$, $\bd = (d_1,\dots,d_\sfp)\in\{1,\dots,\D\}^\sfp$, and $\Ll(X^{\otimes \sfp}\Rr)_{\bn\bd}=\prod_{i=1}^\sfp X_{n_id_i}$. Then, we can rewrite $X^{\otimes \sfp} = \Ll(\Ll(X^{\otimes \sfp}\Rr)_{\bn\bd}\Rr)$ and $A=(A_{\bn l})$. In this notation, we compute
\begin{align*}
    \Ll|\tilde X\Rr|^2 =\Ll|X^{\otimes \sfp}A\Rr|^2=\sum_{l,\,\bn}\sum_{\bd,\,\bd'}\Ll(X^{\otimes \sfp}\Rr)_{\bn\bd}\Ll(X^{\otimes \sfp}\Rr)_{\bn\bd'}A_{\bd l}A_{\bd' l}
\end{align*}
where the summations are over all possible values of these tuples. Let $\Delta_N$ be the collection of tuples $\bn$ where all entries are distinct.
Since rows of $X$ are independent, we have that if $\bn\in\Delta_N$, then $\E \Ll[\Ll(X^{\otimes \sfp}\Rr)_{\bn\bd}\Ll(X^{\otimes \sfp}\Rr)_{\bn\bd'}\Rr]= \prod_{i=1}^\sfp \E \Ll[X_{1d_i}X_{1d'_i}\Rr]$.
Since entries of $X$ are assumed to be in $[-1,+1]$ as in~\ref{i.assume_1_support_X} and $\lim_{N\to\infty}\Ll|\Delta_N\Rr|/N^\sfp =1$, we thus have
\begin{align*}
    N^{-\sfp}\E \Ll|\tilde X\Rr|^2  = o_N(1) + \Ll|\Delta_N\Rr|^{-1}\E \sum_{l,\,\bn\in\Delta_N}\sum_{\bd,\,\bd'}\Ll(X^{\otimes \sfp}\Rr)_{\bn\bd}\Ll(X^{\otimes \sfp}\Rr)_{\bn\bd'}A_{\bd l}A_{\bd' l}
    \\
    =o_N(1)+\sum_{l}\sum_{\bd,\,\bd'}\prod_{i=1}^\sfp \E \Ll[X_{1d_i}X_{1d'_i}\Rr]A_{\bd l}A_{\bd' l}
\end{align*}
which implies~\eqref{e.N^-pE|X|^2}

Next, we turn to~\eqref{e.lim_MMSE}. Similarly, as before, we can compute
\begin{align*}
    N\,  \MMSE_N(t,h) \stackrel{\eqref{e.MMSE=},\eqref{e.<g(x)>=}}{=} \E \Ll[X^\intercal X\Rr] - \E \Ll[\la \bx\ra^\intercal \la\bx\ra\Rr] \stackrel{\eqref{e.derF_N=}}{=}\E \Ll[X^\intercal X\Rr] - N\nabla_h\bar F_N(t,h).
\end{align*}
Then,~\eqref{e.lim_MMSE} follows from~\eqref{e.limdhF_N} and the easy observation that $\E \Ll[X^\intercal X\Rr] = N \E \Ll[X^\intercal_1 X_1\Rr]$ due to the independence of rows in $X$.
\end{proof}



\section{Concentration of overlap}

Recall the overlap matrix $Q$ defined in~\eqref{e.overlap=}.

\begin{proposition}[Concentration of overlap]\label{p.overlap}
Assume~\ref{i.assume_1_support_X}, \ref{i.assume_2_F_N(0,0)_cvg}, and~\ref{i.assume_3_concent}.
If $f(t,\cdot)$ is differentiable at $h\in\S^\D_{++}$ for some $t\in\R_+$, then the following holds:
\begin{itemize}
    \item the averaged overlap converges:
    \begin{align}\label{e.limE<O_N>=}
        \lim_{N\to\infty}\E \la \overlap\ra_{N,t,h} = \nabla_h f(t,h);
    \end{align}

    \item the overlap concentrations:
    \begin{align}\label{e.limE<|O-E<O>|>=0}
    \lim_{N\to\infty} \E \la\Ll|\overlap - \E \la \overlap\ra_{N,t,h} \Rr|\ra_{N,t,h} =0.
    \end{align}
\end{itemize}
Moreover, the limit of $Q$ is related to the maximizers of variational formulas~\eqref{e.hopf_inference} and~\eqref{e.hopf-lax_inference} via Proposition~\ref{p.maximizers}; and to the limit of MMSE via Proposition~\ref{p.mmse}.
\end{proposition}

\begin{remark}
The results in Proposition~\ref{p.overlap} still hold if we replace $\overlap$ given in~\eqref{e.overlap=} by
\begin{align}\label{e.R=}
    R=\frac{1}{N}\bx^\intercal \bx'
\end{align}
where $\bx$ and $\bx'$ are two independent samples from the Gibbs measure $\la\cdot\ra_{N,t,h}$ as in~\eqref{e.<g(x)>=}.
This claim follows from a special case of the well-known Nishimori identity (e.g.\ see~\cite[Proposition 4.1]{HJbook}): for any bounded measurable $g:\R^{N\times \D}\times \R^{N\times \D}\to\R$, we have
\begin{align}\label{e.Nishimori}
    \E \la g(\bx,\bx') \ra_{N,t,h}=\E \la g(\bx,X)\ra_{N,t,h}.
\end{align}
Indeed,~\eqref{e.Nishimori} implies that $\overlap$ and $R$ have the same distribution under $\E \la\cdot\ra_{N,t,h}$. \qed

\end{remark}

To prove~\eqref{e.limE<|O-E<O>|>=0}, we need to compare $\overlap$ with $\cL$ defined by
\begin{align}\label{e.cL=}
    \cL = N^{-1}\nabla_h H_N(t,h,\bx)
\end{align}
where $H_N(t,h,\bx)$ is given in~\eqref{e.H_N(t,h,x)}. We view $\cL$ as a matrix in $\S^\D$ and it is well-defined for $h\in\S^\D_{++}$. Indeed, for $h\in\S^\D_{++}$ and $a\in\S^\D$, we have $h+\eps a\in\S^\D_{++}$ for sufficiently small $\eps$ and its matrix square root exists and is positive definite. Then, we can compute
\begin{align}\label{e.a.L=}
    a\cdot \cL =N^{-1}\frac{\d}{\d \eps}H_N(t,h+\eps a,\bx)\Big|_{\eps=0} = N^{-1}\Ll(\sqrt{2}\cD_{\sqrt{h}}(a)\cdot \bx^\intercal Z +  2a\cdot \bx^\intercal X-a\cdot \bx^\intercal \bx\Rr)
\end{align}
where $\cD_{\sqrt{h}}(a)$ is the derivative of the function $h\mapsto \sqrt{h}$ at $h$ along the direction of $a$, namely,
\begin{align}\label{e.D_sqrt(h)(a)}
    \cD_{\sqrt{h}}(a) = \lim_{\eps\to0}\eps^{-1}\Ll(\sqrt{h+\eps a}-\sqrt{h}\Rr).
\end{align}
We recall from~\cite[(3.7)]{mourrat2020hamilton} a useful estimate:
\begin{align}\label{e.|D_sqrt(h)(a)|<}
    \Ll|\mathcal{D}_{\sqrt{h}}(a)\Rr|\leq C|a|\Ll|h^{-1}\Rr|^\frac{1}{2}
\end{align}
where $C>0$ is an absolute constant and $h^{-1}$ is the matrix inverse of $h\in\S^\D_{++}$.

As in~\cite{barbier2021overlap}, it is more convenient to work with $\cL$ instead of $\overlap$ because $\cL$ is closely related to the free energy via:
\begin{align}\label{e.<L>=nabla_hF_N}
    \la \cL\ra_{N,t,h} = \nabla_h  F_N(t,h),\qquad \E \la \cL\ra_{N,t,h} = \nabla_h \bar F_N(t,h).
\end{align}
Indeed, from~\eqref{e.F_N=} and~\eqref{e.cL=}, we can easily deduce~\eqref{e.<L>=nabla_hF_N}.
We will first derive the concentration of $\cL$ and then relate it to that of $\overlap$ via the following results extractable from~\cite{barbier2021overlap}. 

\begin{lemma}[\cite{barbier2021overlap} Relation between $\overlap$ and $\cL$]
\label{l.QandL}
Assume~\ref{i.assume_1_support_X}.
Let $t\in\R_+$ and $h\in \S^\D_{++}$. 
For brevity, write $\la\cdot\ra=\la \cdot\ra_{N,t,h}$ and set
\begin{align}\label{e.ell(N)=}
    \ell_0(N) = \Ll(\E \la\Ll|\cL-\la \cL\ra\Rr|^2\ra\Rr)^\frac{1}{2},\qquad \ell_1(N) = \Ll(\E \la\Ll|\cL-\E\la \cL\ra\Rr|^2\ra\Rr)^\frac{1}{2}.
\end{align}
Then, there is a constant $C$ such that, for every $N\in\N$,
\begin{gather}
    \E \la \Ll|\overlap- \la\overlap \ra\Rr|^2 \ra \leq C\ell_0(N)+ CN^{-\frac{1}{2}}, \label{e.QandL(1)}
    \\
    \E \la \Ll|\overlap- \la R \ra\Rr|^2 \ra \leq C\ell_0(N)+ CN^{-\frac{1}{2}}, \label{e.QandL(2)}
    \\
    \E \la \Ll|\overlap- \E\la\overlap \ra\Rr|^2\ra \leq C\ell_0(N)^\frac{1}{2}+C\ell_1(N)+CN^{-\frac{1}{4}}. \label{e.QandL(3)}
\end{gather}
\end{lemma}
Here, the constant $C$ depends on $\D$, $h$.

\begin{proof}
These results are contained in~\cite{barbier2021overlap}. The setting in~\cite[Section~2.1]{barbier2021overlap} is more general and covers our current setting. There is an additional factor $-\frac{1}{2}$ in front of the Hamiltonian~\cite{barbier2021overlap} (see display (3.4) therein) compared with~\eqref{e.H_N(t,h,x)}. Also, $\cL$ therein (see (4.1) and (4.2)) are defined as differentiation of the Hamiltonian with respect to each entry of $h$. This is related to our definition in~\eqref{e.a.L=} through choosing the testing matrix $a$ to be $1$ in certain entries and $0$ otherwise. However, in~\cite{barbier2021overlap}, diagonal entries and off-diagonal ones in $\cL$ are weighted differently, resulting in a difference by a factor of $2$. In conclusion, each entry in our $\cL$ is equal to that of $\cL$ in~\cite{barbier2021overlap} up to a constant factor $-2^c$ with $c\in \Z$ depending only on the position of the entry.

In terms of overlaps, $\overlap$ are defined in the same way (see \cite[(2.2)]{barbier2021overlap}). 
Here, $R$ in~\eqref{e.R=} is equal to $Q^{(12)}$ in~\cite[(3.9)]{barbier2021overlap}.
Factors $K$, $S$, $n$ there correspond to $\D$, $1$, $N$ here.

Notice that the magnitude of the derivative of $\sqrt{h}$ blows up when $h$ approaches $0$. In~\cite{barbier2021overlap}, this is controlled by the parameter $s_n$ which determines the range for $h$ (denoted by $\lambda_n$; see (3.2) there). We absorb $s_n$ into the constant $C$ since $h$ is fixed.
Results in~\cite{barbier2021overlap} hold with an additional local average $\E_\lambda$ of $\lambda$ (corresponding to $h$ here). There, this is needed to have the concentration of $\cL$. In this lemma, since we encode the bounds in terms of $\ell_0(N)$ and $\ell_1(N)$, we do not need this local average. Computations in~\cite{barbier2021overlap} often have $\E_\lambda$ but they are still valid without it.


With the above clarification, we are ready to describe how to extract the announced results.
Estimates in~\eqref{e.QandL(1)} and~\eqref{e.QandL(2)} correspond to (3.10) and (3.11) in~\cite[Theorem~3.1]{barbier2021overlap}. We can extract~\eqref{e.QandL(1)} from the first display below~\cite[(4.13)]{barbier2021overlap}. We absorb the factors $C(K,S)$ and $s_n$ therein into $C$ here. 
We can extract~\eqref{e.QandL(2)} from the first display below (4.19) in~\cite{barbier2021overlap} together with the ensuing sentences. The estimate in~\eqref{e.QandL(3)} is contained in the proof of~\cite[Theorem~3.2]{barbier2021overlap}. In the proof, whenever \cite[(3.10) and (3.11)]{barbier2021overlap} are used, we need to replace them by~\eqref{e.QandL(1)} and~\eqref{e.QandL(2)}. This way, we replace every instance of $\frac{C(K,S)}{(s_nn)^\frac{1}{4}}$ and $\mathcal{O}_{K,S}((s_nn)^{-\frac{1}{4}})$ in~\cite{barbier2021overlap} by $(\ell_0(N)+N^{-\frac{1}{2}})^\frac{1}{2}$.
Then,~\eqref{e.QandL(3)} follows from \cite[(4.23) and the next display]{barbier2021overlap}.
\end{proof}

The concentration of $\cL$ is achieved in~\cite{barbier2021overlap} by applying a local average of the additional field, which is a standard technique for spin glass models. Also, it is well-known in spin glass that differentiability implies concentration~\cite{chatterjee}. The next result is of the same flavor.

\begin{lemma}[Concentration of $\cL$]\label{l.concent_cL}
Assume~\ref{i.assume_1_support_X}, \ref{i.assume_2_F_N(0,0)_cvg}, and~\ref{i.assume_3_concent}.
If $f(t,\cdot)$ is differentiable at $(t,h)\in\R_+\times \S^\D_{++}$, then
\begin{align}\label{e.concent.cL}
    \lim_{N\to\infty} \E \la\Ll|\cL - \E \la\cL\ra_{N,t,h}\Rr|\ra_{N,t,h}=0.
\end{align}
\end{lemma}

To prove this lemma, we need to recall the following estimates from~\cite{mourrat2020hamilton}. For every $h\in \S^\D_{++}$, recall that we denote by $h^{-1}$ its matrix inverse.
Also recall the noise matrix $Z$ from~\eqref{e.barY=Xsqrt(2h)+Z}. Henceforth, we write $\nabla = \nabla_h$ for brevity.

\begin{lemma}[\cite{mourrat2020hamilton}]\label{l.nabla^2F}
Assume~\ref{i.assume_1_support_X}.
There is a constant $C>0$ such that, for every $N\in\N$, $t\in\R_+$, $h\in\S^\D_{++}$ and $a\in \S^\D$,
\begin{gather}
    a\cdot\nabla\Ll(a\cdot\nabla  F_N(t,h)\Rr)\geq -CN^{-\frac{1}{2}}|a|^2|Z|\Ll|h^{-1}\Rr|^\frac{3}{2}\label{e.nabla^2F_N>},
    \\
    a\cdot\nabla\Ll(a\cdot\nabla \bar F_N(t,h)\Rr)\geq N\E \la \Ll(a\cdot \cL-\la a\cdot \cL\ra\Rr)^2 \ra_{N,t,h} - C|a|^2\Ll|h^{-1}\Rr|.\label{e.nabla^2barF_N>}
\end{gather}
\end{lemma}
\begin{proof}
Our $\bar F_N(t,h)$ and $Na\cdot\cL$ (see~\eqref{e.F_N=} and~\eqref{e.cL=}) correspond to $\bar F_N(t,2h)$ and $2H'_N(a,h,x)$ in~\cite{mourrat2020hamilton} (see~(1.6) and the display after~(3.18)). We remark that the interaction structure of signals considered in~\cite{mourrat2020hamilton} is quadratic, which is a special case of our setting (see the display above~(1.6)). However, the structure of external fields is the same. Since the computations to be recalled below are carried with respect to $h$, they still hold in our setting.

The estimate in~\eqref{e.nabla^2F_N>} is exactly the second bound in~\cite[(3.37)]{mourrat2020hamilton}.
The estimate in~\eqref{e.nabla^2barF_N>} follows from the combination of~\cite[(3.22) and (3.7)]{mourrat2020hamilton} together with the fact that entries of $\bx$ lie in $[-1,+1]$.
\end{proof}

\begin{proof}[Proof of Lemma~\ref{l.concent_cL}]
Fix $(t,h)$ as in the statement. Since $N$ is clear from the context, we simply write $\la\cdot\ra_{h'}=\la\cdot\ra_{N,t,h'}$ for any $h'\in \S^\D_+$. Also, we write $ F_N(h')= F_N(t,h')$ and $f(h')=f(t,h')$ for any $h'\in\S^\D_+$.

Since $\cL$ is a symmetric matrix, \eqref{e.concent.cL} is equivalent to 
\begin{align}\label{e.E<|g-E<g>|>=0}
    \lim_{N\to\infty} \E \la \Ll| \cL_a - \E \la \cL_a\ra_h\Rr|  \ra_h =0
\end{align}
for every $a\in\S^\D$, where we used the shorthand notation $\cL_a = a\cdot \cL$. 
In fact, we can take $a$ from the orthogonal basis of $\S^\D$ consisting of matrices with only $0$ or $1$ entries. In this way, \eqref{e.E<|g-E<g>|>=0} becomes the concentration of each entry of $\cL$.
Henceforth, we fix an arbitrary $a$ and verify~\eqref{e.E<|g-E<g>|>=0}. We proceed in two steps.

Step~1. We show
\begin{align}\label{e.E<|g-<g>|>=0}
    \lim_{N\to\infty}\E \la \Ll|\cL_a - \la \cL_a\ra_h\Rr|  \ra_h =0.
\end{align}
We denote by $(\bx^l)_{l\in\N}$ independent copies of $\bx$ under $\la \cdot\ra_h$ and let $\cL_a^l$ to be $\cL_a$ with $\bx$ therein (see~\eqref{e.a.L=}) replaced by $\bx^l$. Fix $r_0>0$ sufficiently small so that $h+sa\in \S^\D_{++}$ for all $s\in [-r_0,r_0]$.
For $r\in(0,r_0]$, integrating by parts, we have
\begin{align*}
    r\E \la \Ll|\cL_a^1 -\cL_a^2\Rr|\ra_h = \int_0^r \E \la \Ll|\cL_a^1 -\cL_a^2\Rr|\ra_{h+sa} \d s
    -\int_0^r \int_0^\tau \frac{\d}{\d s} \E \la \Ll|\cL_a^1 -\cL_a^2\Rr|\ra_{h+sa}\d s \d \tau.
\end{align*}
Using~\eqref{e.cL=} and~\eqref{e.<>=}, we can compute 
\begin{align*}
    \frac{\d}{\d s} \E \la \Ll|\cL_a^1 -\cL_a^2\Rr|\ra_{h+sa}  = N \E \la \Ll|\cL_a^1 -\cL_a^2\Rr|\Ll(\cL_a^1+\cL_a^2-2\cL_a^3\Rr)\ra_{h+sa}
    \\
    \geq -2N \E \la \Ll|\cL_a^1 -\cL_a^2\Rr|^2\ra_{h+sa}\geq -8N \E \la \Ll|\cL_a -\la \cL_a\ra_{h+sa}\Rr|^2\ra_{h+sa}.
\end{align*}
Combining the above two displays, we get
\begin{align}
    \E \la \Ll|\cL_a^1 -\cL_a^2\Rr|\ra_h &\leq \frac{1}{r}\int_0^r\E \la \Ll|\cL_a^1 -\cL_a^2\Rr|\ra_{h+sa} \d s 
    + \frac{8N}{r}\int_0^r \int_0^\tau \E \la \Ll|\cL_a -\la \cL_a\ra_{h+sa}\Rr|^2\ra_{h+sa}\d s \d \tau \notag
    \\
    &\leq \frac{2}{r}\int_0^r\E \la \Ll|\cL_a -\la \cL_a\ra_{h+sa}\Rr|\ra_{h+sa} \d s \label{e.app_CS}
    + 8N\int_0^r \E \la \Ll|\cL_a -\la \cL_a\ra_{h+sa}\Rr|^2\ra_{h+sa}\d s .
\end{align}
Setting
\begin{align*}
    \eps_N = N \int_0^r \E \la \Ll|\cL_a -\la \cL_a\ra_{h+sa}\Rr|^2\ra_{h+sa}\d s
\end{align*}
and applying the Cauchy--Schwarz inequality to the first integrand in~\eqref{e.app_CS},
we get
\begin{align}\label{e.E<g-g><2..eps..}
    \E \la \Ll|\cL_a^1 -\cL_a^2\Rr|\ra_h \leq 2 \sqrt{\frac{\eps_N}{rN}}+ 8 \eps_N.
\end{align}
By~\eqref{e.nabla^2barF_N>} from Lemma~\ref{l.nabla^2F}, there is a constant $C>0$ depending only on $h$, $a$, and $r_0$ such that
\begin{align*}
    N\E \la \Ll|\cL_a-\la \cL_a\ra_{h+sa}\Rr|^2\ra_{h+sa}\leq (a\cdot\nabla )\Ll(a\cdot\nabla\bar F_N(h+sa)\Rr)+ C
\end{align*}
for every $s\in [0,r_0]$.
Hence, we have
\begin{align*}
    \eps_N & \leq \int_0^r \frac{\d^2}{\d s^2}\bar F_N(h+sa) \d s+Cr =  a\cdot\nabla \bar F_N(h+ra) - a\cdot\nabla \bar F_N(h) +Cr
    \\
    &\leq \frac{\bar F_N(h+(r+\tau)a) - \bar F_N(h+ra)}{\tau} - \frac{\bar F_N(h) - \bar F_N(h-\tau a)}{\tau}+Cr
\end{align*}
for any $\tau\in(0,r_0]$, where the last inequality follows from the convexity of $\bar F_N$ (see~\cite[Lemma~2.3]{chen2022statistical}).
Therefore, the above display along with~\eqref{e.E<g-g><2..eps..} and the convergence of $\bar F_N$ to $f$ given by Theorem~\ref{t.cvgF_N_to_f} implies
\begin{align*}
    \limsup_{N\to\infty} \frac{1}{8}\E \la \Ll|\cL_a^1 -\cL_a^2\Rr|\ra_h \leq \frac{ f(h+(r+\tau)a) -  f(h+ra)}{\tau} - \frac{ f(h) -  f(h-\tau a)}{\tau} + Cr
\end{align*}
We first send $r\to0$ and then $\tau\to 0$. By the differentiability assumption on $f$, the right-hand side becomes zero. This immediately yields~\eqref{e.E<|g-<g>|>=0}. 

Step~2. We show
\begin{align}\label{e.E|<g>-E<g>|=0}
    \lim_{N\to\infty}\E\Ll| \la \cL_a\ra_h - \E\la \cL_a\ra_h\Rr|  =0.
\end{align}
We start with~\eqref{e.<L>=nabla_hF_N} which gives
\begin{align}\label{e.1/NE|g-g|=...}
     \E\Ll| \la \cL_a\ra_h - \E\la \cL_a\ra_h\Rr| = \E \Ll| a\cdot\nabla F_N(h) - a\cdot\nabla \bar F_N(h)\Rr|.
\end{align}
Recall that $r_0>0$ is fixed so that $h+sa\in \S^\D_{++}$ for every $s\in [-r_0,r_0]$. 
For $r\in(0,r_0]$, we set
\begin{align*}
    \delta_N(r) = \Ll|F_N(h-ra)-\bar F_N(h-ra)\Rr| + \Ll|F_N(h)-\bar F_N(h)\Rr| + \Ll|F_N(h+ra)-\bar F_N(h+ra)\Rr|.
\end{align*}
By~\eqref{e.nabla^2F_N>} from Lemma~\ref{l.nabla^2F} and the Taylor expansion, there is a constant $C$ depending only on $h$, $a$, and $r_0$ such that, for every $r\in(0,r_0]$,
\begin{align*}
    F_N(h+ra)-F_N(h) \geq ra\cdot\nabla F_N(h) -Cr^2N^{-\frac{1}{2}}|Z|,
    \\
    F_N(h-ra)-F_N(h) \leq ra\cdot\nabla F_N(h) +Cr^2N^{-\frac{1}{2}}|Z|.
\end{align*}
The above two displays together give
\begin{align*}
    a\cdot\nabla F_N(h) - a\cdot\nabla \bar F_N(h) \leq \frac{\bar F_N(h+ra)-\bar F_N(h)}{r}-a\cdot \nabla \bar F_N(h)+ \frac{\delta_N(r)}{r} + CrN^{-\frac{1}{2}}|Z|,
    \\
    a\cdot\nabla F_N(h) - a\cdot\nabla \bar F_N(h) \geq \frac{\bar F_N(h)-\bar F_N(h-ra)}{r}-a\cdot \nabla \bar F_N(h)- \frac{\delta_N(r)}{r}-CrN^{-\frac{1}{2}}|Z|.
\end{align*}
The concentration assumption~\ref{i.assume_3_concent} ensures $\lim_{N\to\infty}\E \delta_N(r)=0$.
Recall from~\eqref{e.barY=Xsqrt(2h)+Z} that $Z$ is an $N\times \D$ matrix with i.i.d.\ standard Gaussian entries. Therefore, we have $\E |Z|\leq \sqrt{N\D}$. Also, $\bar F_N$ converges to $f$ pointwise as given by Theorem~\ref{t.cvgF_N_to_f}.
Using these and~\eqref{e.limdhF_N}, we get
\begin{align*}
    \limsup_{N\to\infty} \E \Ll|a\cdot\nabla F_N(h) - a\cdot\nabla \bar F_N(h)\Rr| 
    &\leq \Ll|\frac{ f(h+ra)- f(h)}{r}-a\cdot \nabla  f(h)\Rr| 
    \\
    &+ \Ll|\frac{ f(h)- f(h-ra)}{r}-a\cdot \nabla  f(h)\Rr| + Cr\sqrt{\D}.
\end{align*}
Now, sending $r\to0$ and using the differentiability of $f$ at $h$, we deduce that the left-hand side in the above vanishes. Finally, inserting this to \eqref{e.1/NE|g-g|=...}, we arrive at~\eqref{e.E|<g>-E<g>|=0}.

Since~\eqref{e.E<|g-<g>|>=0} and~\eqref{e.E|<g>-E<g>|=0} together yield~\eqref{e.E<|g-E<g>|>=0}, the proof is complete.
\end{proof}

Now, we are ready for the main task.

\begin{proof}[Proof of Proposition~\ref{p.overlap}]
Recall $Q$ from~\eqref{e.overlap=}.
Fix any $(t,h)\in \R_+\times \S^\D_{++}$ such that $f(t,\cdot)$ is differentiable at $h$. 
We verify~\eqref{e.limE<O_N>=} and~\eqref{e.limE<|O-E<O>|>=0}. First, \eqref{e.limE<O_N>=} follows from~\eqref{e.limdhF_N}, \eqref{e.derF_N=}, \eqref{e.Nishimori}. Then, we turn to~\eqref{e.limE<|O-E<O>|>=0}.
Recall $\ell_0(N)$ and $\ell_1(N)$ from~\eqref{e.ell(N)=}. It is straightforward to see $\ell_0(N)\leq 2\ell_1(N)$. Hence, \eqref{e.limE<|O-E<O>|>=0} follows from~\eqref{e.QandL(3)} in Lemma~\ref{l.QandL} provided $\lim_{N\to\infty}\ell_1(N)=0$. In the following, we deduce this from Lemma~\ref{l.concent_cL}.
Using~\eqref{e.a.L=}, \eqref{e.|D_sqrt(h)(a)|<}, and assumption~\ref{i.assume_1_support_X}, there is a constant depending only on $h$ and $\D$ such that
\begin{align*}
    \Ll|a \cdot \cL\Rr|\leq C\Ll(N^{-\frac{1}{2}}|Z|+1\Rr)|a|
\end{align*}
which immediately implies
\begin{align*}
    \Ll|\cL\Rr|\leq C\Ll(N^{-\frac{1}{2}}|Z|+1\Rr).
\end{align*}
We write $\la\cdot\ra=\la\cdot\ra_{N,t,h}$ for simplicity.
Using this, we have that, for any $r>0$,
\begin{align*}
    \E \la\Ll|\cL-\E\la \cL\ra\Rr|^2\ra  
    = \E\Ll[ \la\Ll|\cL-\E\la \cL\ra\Rr|^2\ra\mathbf{1}_{N^{-\frac{1}{2}}|Z|\leq r}\Rr] + \E\Ll[ \la\Ll|\cL-\E\la \cL\ra\Rr|^2\ra\mathbf{1}_{N^{-\frac{1}{2}}|Z|> r}\Rr]
    \\
    \leq 2C(r+1)\E \la \Ll|\cL-\E\la \cL\ra\Rr|\ra + 8C^2\E\Ll[\Ll(N^{-1}|Z|^2+1\Rr)\mathbf{1}_{N^{-\frac{1}{2}}|Z|> r}\Rr]
\end{align*}
Since $Z$ is an $N\times \D$ matrix with i.i.d.\ standard Gaussian entries (see~\eqref{e.barY=Xsqrt(2h)+Z}), it is standard (e.g.\ \cite[Theorem~3.1.1]{vershynin2018high}) to see that for every $\eps>0$, there is $r>0$ such that the last term is bounded by $\eps$. Then, applying Lemma~\ref{l.concent_cL}, we get
\begin{align*}
    \limsup_{N\to\infty} \E \la\Ll|\cL-\E\la \cL\ra\Rr|^2\ra \leq \eps.
\end{align*}
Sending $\eps\to0$, we can thus deduce $\lim_{N\to\infty} \ell_1(N)=0$ and \eqref{e.limE<|O-E<O>|>=0}.
\end{proof}

\section{Short-time regularity}

Recall that $\psi$ is the initial condition given in~\ref{i.assume_2_F_N(0,0)_cvg}. In this section, we assume that
\begin{align}\label{e.psi_C_2}
    \text{$\psi$ is twice differentiable with bounded derivatives (both first and second order),}
\end{align}
as assumed in the last part of Theorem~\ref{t}. First, we show that this condition is satisfied under the stronger assumption~\ref{i.assume_S}. Note that under the condition in~\eqref{e.psi_C_2} the gradient of $\psi$ is a Lipschitz function. In this section, again, we write $\nabla =\nabla_h$ for brevity.

\begin{lemma}\label{l.HS_condition}
Under~\ref{i.assume_S}, the condition in ~\eqref{e.psi_C_2} holds.
\end{lemma}

\begin{proof} 
Under~\ref{i.assume_S}, we have $\psi = \bar F_1(0,\cdot)$. As previously, we write $\langle \cdot \rangle$ in place of $\langle \cdot \rangle_{N=1,t=0,h}$. We let $\bx'$ and $\bx''$ denote independent copies of $\bx$ under $\langle \cdot \rangle$. Then, we can directly compute the derivatives of $\bar F_1(0,\cdot)$ using Gaussian integration by parts (as proven in \cite[(3.8) and (3.10)]{mourrat2020hamilton}) and the Nishimori identity~\eqref{e.Nishimori} (also \cite[(3.3)]{mourrat2020hamilton}). More precisely we will use the following identities, which hold for any $\R$-valued bounded measurable function $g$ and any $\R^{N \times D}$-valued bounded measurable function $G$,
\begin{align*} 
    \E \la g(x,x') \ra &= \E \la g(x,X) \ra, \\
    \E \la Z \cdot G(\bx,X) \ra &= \E \la (\bx-\bx'')\sqrt{2h} \cdot G(\bx,X)  \ra, \\
    \E \la Z \cdot G(\bx,\bx',X) \ra &= \E \la (\bx+\bx'-2\bx'')\sqrt{2h} \cdot G(\bx,\bx',X)  \ra. 
\end{align*}
For $a \in S^D$, we denote by $a \cdot \nabla$ the operator defined by taking the directional derivative in direction $a$ with respect to $h$, that is
\begin{equation*}
    (a \cdot \nabla) g(h) = \lim_{\varepsilon \to 0} \frac{g(h+\varepsilon a) - g(h)}{\varepsilon}.
\end{equation*}
Recall from~\eqref{e.D_sqrt(h)(a)} that $\mcl D_{\sqrt{h}}(a)$ denotes the derivative in the direction $a$ of the square root function at $h$.
Then, we have
\begin{equation*}
    (a \cdot \nabla) H_1(0,h,\bx) = \sqrt{2} \mcl D_{\sqrt{h}}(a) \cdot \bx^\intercal Z + 2a \cdot \bx^\intercal X - a \cdot \bx^\intercal \bx
\end{equation*}
Differentiating $\psi$ in direction $a$ we obtain
\begin{equation*}
    (a \cdot \nabla) \psi(h) = \E \langle(a \cdot \nabla) H_1(0,h,\bx) \rangle.
\end{equation*}
Using the first Gaussian integration by part identity as recalled above, we recover \cite[(2.2)]{chen2022statistical}, namely
\begin{equation*}
    (a \cdot \nabla) \psi(h) = \E \langle a \cdot \bx^\intercal X \rangle.
\end{equation*}
The boundedness of the signal yields that $\left|(a \cdot \nabla) \psi(h) \right| \leq C|a|$ and thus the first derivative of $\psi$ is bounded. Furthermore, differentiating the expression in the previous display in direction $b \in S^D$ yields
\begin{equation*}
   (b \cdot \nabla) (a \cdot \nabla) \psi(h) = \E \la a \cdot \bx^\intercal X \left( (b \cdot \nabla) H_1(0,h,\bx) - (b \cdot \nabla) H_1(0,h,\bx') \right)\ra.
\end{equation*}
Again using the Gaussian integration by part identities we obtain, like in \cite[(3.27)]{mourrat2020hamilton}, an expression of the form 
\begin{equation*}
   (b \cdot \nabla) (a \cdot \nabla) \psi(h) = \E \langle \left( a \cdot \bx^\intercal X \right) b \cdot P(\bx,\bx',\bx'',X) \rangle,
\end{equation*}
where $P$ is some explicit homogeneous polynomial of degree 2. Again, from the boundedness of the signals, we see that 
\begin{equation*}
    |(b \cdot \nabla) (a \cdot \nabla) \psi(h)| \leq C|a||b|.
\end{equation*}
This proves that the second derivative of $\psi$ is bounded.
\end{proof}
We consider the nonlinearity $\H$ as in~\eqref{e.H(q)=(AA)...}. For each $h\in\S^\D_+$, the characteristic line $t\mapsto X(t,h)$ emitting from $h$ is defined as 
\begin{align*}
    X(t,h) = h - t\nabla\H(\nabla\psi(h)),\quad\forall t\in \R_+.
\end{align*}
Under~\eqref{e.psi_C_2}, the goal is to show that the solution of~\eqref{e.hj_stats} is twice differentiable for a short time. For equations defined on the entire Euclidean space, this is a standard result \cite[Exercise~2.10]{HJbook}. But in our case, since the domain is the convex cone $\S^\D_+$, this standard argument does not apply directly. We need to ensure that the characteristics in the backward direction do not leave $\S^\D_+$.
Fortunately, we have
\begin{align*}
    \nabla \H(a) \in \S^\D_+,\quad\forall a\in\S^\D_+
\end{align*}
which is a result of the monotonicity of $\H$ proved in~\cite[Lemma~4.2]{chen2022statistical}. Also, $\nabla\psi$ takes value in $\S^\D_+$ since $\bar F_N(0,\cdot)$ is increasing in the direction of $\S^\D_+$ as $\nabla\bar F_N(0,h)\in\S^\D_+$ (see~\cite[Lemma~2.1]{chen2022statistical}) and $\psi$ is the limit of $\bar F_N(0,\cdot)$. Hence, we always have $\nabla\H(\nabla\psi(h))\in\S^\D_+$, which means that the characteristics go further into the cone $\S^\D_+$ when traced backward.

Recall $L$ from~\eqref{e.L=}.
We first show that when $t<L^{-1}$ the characteristics form a diffeomorphism.

\begin{lemma}\label{l.X(t,.)properties}
Assume~\eqref{e.psi_C_2}. Let $t\in [0, L^{-1})$, then the following holds.
\begin{enumerate}
    \item\label{i.l.X(t,.)properties_1} The map $X(t,\cdot):\S^\D_+\to X(t,\S^\D_+)$ admits a differentiable inverse function $Z(t,\cdot)$.
    \item\label{i.l.X(t,.)properties_2} We have $X(t,\S^\D_+)\supset \S^\D_+$.
\end{enumerate}
\end{lemma}

By an inverse function, we mean $X(t,Z(t,h))=h$ for every $h\in\S^\D_+$ here.

\begin{proof}
Due to the definition of $L$ and $t<L^{-1}$, we can see that $\det(\nabla X(t,h))>0$ for every $h\in\S^\D_+$. Therefore, in each local neighborhood, $X(t,\cdot)$ is invertible. The definition of $L$ and $t<L^{-1}$ also ensure that the preimage of any bounded set under $X(t,\cdot)$ is still bounded. Therefore, we can resort to classical results on the invertibility of maps (e.g.\ \cite[Theorem~1.8 in Chapter~3]{ambrosetti1995primer}) to get the existence of a continuous map $Z(t,\cdot)$ that serves as the inverse of $X(t,\cdot)$. Again using $t<L^{-1}$, we can first verify that $Z(t,\cdot)$ is Lipschitz. Then, using the smoothness of $X(t,\cdot)$ and the invertibility condition, we can verify that $Z(t,\cdot)$ is differentiable everywhere. This verifies Part~\eqref{i.l.X(t,.)properties_1}.

We show that for each $k\in\S^\D_+$ there is $h$ such that  $X(t,h)=k$.
Let $\mfk C_{t,k} \subset \S^\D_+$ denote the closed convex envelope of $\Ll\{k + t\nabla\H(\nabla\psi(h)) \big| \, h \in\S^\D_+\Rr\}$ and let 
\begin{equation*}
    \Phi_{t,k} : \begin{cases}
                \mfk C_{t,k} \longrightarrow \mfk C_{t,k} \\
                h \longmapsto k + t\nabla\H(\nabla\psi(h))
              \end{cases}.
\end{equation*}
Since $\Phi_{t,k}$ is continuous and $\mfk C_{t,k}$ is a compact convex set (as $\nabla\psi$ is bounded), it follows from Brouwer's fixed point theorem that $\Phi_{t,k}$ admits a fixed point $h^* \in \mfk C_{t,k}$ which satisfies $k = h^* - t\nabla\H(\nabla\psi(h^*)) = X(t,h^*)$. Hence, Part~\eqref{i.l.X(t,.)properties_2} is valid.
\end{proof}

Lemma~\ref{l.X(t,.)properties} allows us to consider the restriction of $Z(t,\cdot)$ to $\S^\D_+$. Then, we can proceed via the standard argument of characteristics to show the existence of a smooth solution for a short time.

\begin{proposition}\label{p.C^2}
Assume~\eqref{e.psi_C_2}.
Let $L$ be given as in~\eqref{e.L=} and let $f$ be the unique Lipschitz viscosity solution of~\eqref{e.hj_stats} with initial condition $f(0,\cdot)=\psi$. Then, the restriction of $f$ to $[0,L^{-1})\times \S^\D_+$ is twice differentiable everywhere.
\end{proposition}

The proof is classical but we choose to include it here for completeness.

\begin{proof}
It is sufficient to construct a twice-differentiable Lipschitz increasing solution $u$ of~\eqref{e.hj_stats} on $[0,L^{-1})\times \S^\D_+$ with $u(0,\cdot)=\psi$. Indeed, we can conclude that $f$ coincide with $u$ on the prescribed domain from the uniqueness of solution on $[0,L^{-1})\times \S^\D_+$ (see \cite[Corollary~3.2]{chen2022hamiltonCones} which can be easily adapted to the domain $[0,L^{-1})\times \S^\D_+$ from $\R_+\times \S^\D_+$).

With this explained, we turn to the construction of $u$. 
We define
\begin{align*}
    U(t,h)=\psi(h)-t \nabla\H(\nabla\psi(h))\cdot \nabla \psi(h)+ t\H(\nabla\psi(h)),\quad\forall (t,h) \in \R_+\times \S^\D_+.
\end{align*}
Let $Z$ be given by Lemma~\ref{l.X(t,.)properties} and we can restrict it to $[0,L^{-1})\times \S^\D_+$. Then, we set
\begin{align*}
    u(t,h) = U(t,Z(t,h)),\quad\forall (t,h) \in [0, L^{-1})\times \S^\D_+.
\end{align*}
We first verify the initial condition. Since $U(0,\cdot) =\psi$ and $Z(0,\cdot)$ is the identity map, we get $u(0,\cdot)=\psi$ as desired.

From the definition of $u$, it is clear that $u$ is differentiable everywhere. Then, we verify that $u$ satisfies the equation~\eqref{e.hj_stats}, the procedure of which will also show that $u$ is twice-differentiable.
To carry out computations, we introduce some notation. Let $d = \D(\D-1)/2$ and we fix any orthogonal basis $\{e_1,\dots, e_d\}$ for $\S^\D$. For any suitable function $g$, we write $g_t$ for the derivative in $t$ and $g_i$ for the directional derivative in $h$ along $e_i$, for $i\in\{1,\dots ,d\}$.  

For convenience, we identify the linear inner product space $\S^\D$ with $\R^d$ and think of $\S^\D_+$ as a subset of $\R^d$. This allows us to think of the gradient of a scalar function $g$ (such as $\psi$, $\H$, $U$, and $u$) as a column vector $\nabla g= (g_i)_{1\leq i\leq d}$ in $\R^d$; its Hessian $\nabla^2 g=(g_{ij})_{1\leq i,j\leq d}$ as a symmetric $d\times d$ matrix; an $\R^d$-valued function $G = (G^i)_{1\leq i\leq d}$ (such as $X$ and $Y$) as a column vector; and its gradient $\nabla G = (G^i_j)_{1\leq i,j\leq d}$ as a $d\times d$ matrix. In the following, we view $\cdot$ as the inner product in $\R^d$ and always evaluate matrix multiplication (including matrix multiplied by a vector) before evaluating the inner product.
For $(t,h)$ clear from the context, we also write $\nabla g^Z$ or $g_t^Z$ to indicate that we evaluate $\nabla g$ or $g_t$ at the spatial variable $Z(t,h)$.

In the following computations, we keep $(t,h)\in [0,L^{-1})\times \S^\D_+$ implicit.
We first compute the spatial derivative of $u$. We start with some basic computations.
\begin{align}
    \nabla U &= (U_i)_{1\leq i\leq d} = \nabla \psi - t\Ll(\nabla^2\H(\nabla\psi)\nabla^2\psi\Rr)^\intercal \nabla \psi,\notag
    \\
    \nabla X & = (X^i_j)_{1\leq i,j\leq d} = \mathbf{I} - t\nabla^2H(\nabla\psi)\nabla^2 \psi,\notag
    \\
    \nabla u & = (\nabla Z)^\intercal \nabla U^Z. \label{e.nablau=nablaZnablaU}
\end{align}
where $\mathbf{I}$ is the $d\times d$ identity matrix. The first two relations give
\begin{align*}
    \nabla U = (\nabla X)^\intercal \nabla \psi.
\end{align*}
This along with~\eqref{e.nablau=nablaZnablaU} and the fact that $Z$ is the inverse of $X$ implies 
\begin{equation} \label{e.nabla_u=}
    \nabla u = \nabla \psi^Z.
\end{equation}
Next, we compute the derivative of $u$ in $t$. We start with
\begin{align}
    U_t &= - \nabla H(\nabla \psi) \cdot \nabla \psi + H(\nabla \psi)  \notag
    \\
    X_t & = -\nabla H(\nabla\psi),\label{e.X_t=}
    \\
    u_t &= U^Z_t + \nabla U^Z \cdot Z_t \notag
\end{align}
Therefore, 
\begin{equation}
    u_t = -\nabla \H\Ll(\nabla\psi^Z\Rr)\cdot \nabla \psi^Z + \H\Ll(\nabla \psi^Z\Rr) + \Ll(\nabla X^Z\Rr)^\intercal \nabla\psi^Z\cdot Z_t. \label{e.u_t=}
\end{equation}
Notice that the last term is of the form $A^\intercal v\cdot v'$ for a matrix $A$ and vectors $v,v'$. We can rearrange it into $Av'\cdot v$.
Since $Z$ is the inverse of $X$, we get
\begin{align*}
    X^Z_t + \nabla X^Z Z_t =0.
\end{align*}
Using this, \eqref{e.X_t=} and the aforementioned rearrangement to cancel the first and third terms in~\eqref{e.u_t=}, we get
\begin{align*}
    u_t = \H\Ll(\nabla\psi^Z\Rr).
\end{align*}
This along with~\eqref{e.nabla_u=} implies that $u$ satisfies the equation~\eqref{e.hj_stats} in the classical sense. Moreover, we can also infer from them that $u$ is twice differentiable everywhere. Finally, from \eqref{e.nabla_u=} and the monotony of $\psi$, we get that $\nabla u \in S^D_+$, so $u(t,\cdot)$ is an increasing function. 
\end{proof}

\smallskip

\noindent \textbf{Funding.} HBC is funded by the Simons Foundation. 

\noindent
\textbf{Data availability.}
No datasets were generated during this work.

\noindent
\textbf{Conflict of interests.}
The authors have no conflicts of interest to declare.

\noindent
\textbf{Competing interests.}
The authors have no competing interests to declare.

%

\small
\bibliographystyle{plain}
\bibliography{ref}

\end{document}